\documentclass [12pt,leqno,fleqn]{amsart}

\usepackage{amsmath,amstext,amsthm,amsxtra,mathtools,amssymb}

\usepackage[colorlinks=true,citecolor=blue,pagebackref,citecolor=magenta, urlcolor=cyan , hypertexnames=true]{hyperref}

\usepackage[normalem]{ulem}
\usepackage{txfonts} 
\usepackage[T1]{fontenc}
\usepackage{lmodern}

 \usepackage{euler}   

\usepackage{tikz} 

\usepackage[backrefs,msc-links,nobysame,initials]{amsrefs}

\allowdisplaybreaks

\usepackage{todonotes}
\makeatletter

\def\R{\mathbb R}

\def\M{\mathsf M}
\def\Mw{\mathsf M _w}
\def\Mmu{\mathsf M _w}

\def\C{\mathsf C}
\def\Cmu{\mathsf C _w}
\def\Cw{\mathsf C^{w}}

\def\norm#1.#2.{\lVert#1\rVert_{#2}}
\def\Norm#1.#2.{\bigl\lVert#1\bigr\rVert_{#2}}
\def\NOrm#1.#2.{\Bigl\lVert#1\Bigr\rVert_{#2}}
\def\NORm#1.#2.{\biggl\lVert#1\biggr\rVert_{#2}}
\def\NORM#1.#2.{\Biggl\lVert#1\Biggr\rVert_{#2}}

\def\ip#1,#2,{\langle #1,#2\rangle}
\def\Ip#1,#2,{\bigl\langle#1,#2\bigr\rangle}
\def\IP#1,#2,{\Bigl\langle#1,#2\Bigr\rangle}

\def\abs#1{\lvert#1\rvert}
\def\Abs#1{\bigl\lvert#1\bigr\rvert}

\theoremstyle{plain}
\newtheorem{theorem}{Theorem}[section]
\newtheorem{corollary}[theorem]{Corollary}
\newtheorem{lemma}[theorem]{Lemma}
\newtheorem{proposition}[theorem]{Proposition}

\theoremstyle{definition}
\newtheorem{definition}[theorem]{Definition}
\newtheorem{remark}[theorem]{Remark}

\numberwithin{equation}{section}

\title[Solyanik Estimates for the Hardy-Littlewood Maximal Operator]{Weighted Solyanik Estimates for the Hardy-Littlewood maximal operator and embedding of $A_\infty$ into $A_p$}

\author{Paul Hagelstein}
\address{Department of Mathematics, Baylor University, Waco, Texas 76798}
\email{\href{mailto:paul_hagelstein@baylor.edu}{paul\!\hspace{.018in}\_\,hagelstein@baylor.edu}}
\thanks{P. H. is partially supported by a grant from the Simons Foundation (\#208831 to Paul Hagelstein).}

\author{Ioannis Parissis}
\address{Department of Mathematics, Aalto University, P. O. Box 11100, FI-00076 Aalto, Finland}
\email{\href{mailto:ioannis.parissis@gmail.com}{ioannis.parissis@gmail.com}}
\thanks{I. P. is supported by the Academy of Finland, grant 138738.}


\begin{document}
	
\begin{abstract}
Let $w$ denote a weight in  $\mathbb{R}^n$ which belongs to the Muckenhoupt class $A_\infty$ and let $\Mmu$ denote the uncentered Hardy-Littlewood maximal operator defined with respect to the measure $w(x)dx$. The \emph{sharp Tauberian constant} of $\Mmu$ with respect to $\alpha$, denoted by $\Cmu(\alpha)$, is defined by
\[
\Cmu(\alpha) \coloneqq \sup_{E:\, 0 < w(E) < \infty}w(E)^{-1}w\big(\big\{x \in \mathbb{R}^n:\, \Mmu \chi_E (x) > \alpha\big\}\big).
\]
In this paper, we show that the Solyanik  estimate
\[
\lim_{\alpha \rightarrow 1^-}\Cmu(\alpha) = 1
\]
holds. Following the classical theme of weighted norm inequalities we also consider the sharp Tauberian constants defined with respect to the  usual uncentered Hardy-Littlewood maximal operator $\M$ and a weight	 $w$:
\[
\Cw(\alpha) \coloneqq \sup_{E:\, 0 < w(E) < \infty}w(E)^{-1} w\big(\big\{x \in \R^n:\, \M \chi_E (x) > \alpha\big\}\big).
\]
 We show that we have $\lim_{\alpha\to 1^{-}}\Cw(\alpha)=1$ if and only if $w\in A_\infty$. As a corollary of our  methods we obtain a quantitative embedding of $A_\infty$ into $A_p$.
\end{abstract}

\maketitle
\section{Introduction}
We are interested in sharp asymptotic estimates for the distribution set of the Hardy-Littlewood maximal operator defined with respect to a measure $\mu$. We immediately restrict our attention to measures $d\mu(x)=w(x)dx$ for some appropriate locally integrable weight $w$ in $\R^n$ and consider the uncentered Hardy-Littlewood maximal operator defined with respect to $w$:
\[
 \Mw  f(x)\coloneqq \sup_{x\in Q} \frac{1}{w(Q)}\int_Q |f(y)|w(y)dy,\quad  x\in \R^n.
\]
Here the supremum is taken over all cubes in $\R^n$ whose sides are parallel to the axes that contain the point $x$. We study in particular the asymptotic behavior of the \emph{sharp Tauberian constant}
\[
 \Cmu(\alpha)\coloneqq \sup_{E:\, 0<w(E)<\infty} w(E)^{-1} w(\{x\in\R^n:\, \Mmu \chi_E (x)>\alpha\})
\]
as $\alpha\to 1^-$. In the case that $w\equiv 1$ is the Lebesgue measure in $\R^n$ we drop the subscript $w$ and use the notation $\M$ and $\C$ for the Hardy-Littlewood maximal operator and its sharp Tauberian constant, respectively. Our main objective is to investigate whether
\begin{equation}\label{e.main}
 \lim_{\alpha\to 1^-} \Cmu(\alpha)=1.
\end{equation}
For the $n$-dimensional Lebesgue measure Solyanik showed in \cite{Solyanik} that we have $\C(\alpha)-1\eqsim_n (1-\alpha)^\frac{1}{n}$ when $\alpha\to 1^-$ and thus \eqref{e.main} holds. Solyanik also showed corresponding estimates for the centered Hardy-Littlewood maximal operator as well as for the strong maximal operator while in \cite{HP} a similar estimate for the Hardy-Littlewood maximal operator with respect to Euclidean balls is proved. Note that estimate \eqref{e.main} is a certain continuity assertion at $\alpha=1$. Continuity properties of $\C(\alpha)$ for $\alpha<1$ are studied in \cite{BH}. We loosely refer to an estimate of the type \eqref{e.main} as a \emph{Solyanik estimate}.

The reader might appreciate a few words regarding the terminology  ``Tauberian constant'' in this context.   In \cite{cf77}, A. C\'ordoba and R. Fefferman introduced a model multiplier operator $T_\theta$ and corresponding maximal operator $\mathsf M_\theta$ (the precise definitions of which do not concern us here.)   They proved that if  $\mathsf M_\theta$ is bounded on $L^{(p/2)'}(\mathbb{R}^2)$ then $T_\theta$ is bounded on $L^p(\mathbb{R}^2)$ for $1 < p < \infty$.   Conversely they showed that if $T_\theta$ is bounded on $L^p(\mathbb{R}^2)$ and $\mathsf M_\theta$ satisfies the additional  assumption  $|\{ \mathsf M_{\theta} \chi_E  > 1/2\}| \lesssim |E|$ for all measurable $E\subseteq\R^2$, that they referred to as a Tauberian condition, then $\mathsf M_\theta$ is of weak type $((p/2)', (p/2)')$.   This terminology was quite appropriate as, aptly worded by Rudin in \cite{Rud}, ``Tauberian theorems are often converses of fairly obvious results, but usually these converses depend on some additional assumption, called a \emph{tauberian condition}.''   In \cite{cf77}, the additional assumption was of course $|\{ \mathsf M_{\theta} \chi_E> 1/2 \}|\lesssim |E|$.   In \cite{hs}, Hagelstein and Stokolos investigated the growth of $\sup_E  |E|^{-1}|\{\mathsf M_{\mathcal{B} } \chi_E    > \alpha \}|$ as $\alpha \rightarrow 0$ for a very general class of maximal operators $\mathsf M_\mathcal{B}$.    Reflecting the influence of C\'ordoba and Fefferman's paper,  they said that if  $\sup_E |E|^{-1}\{ M_{\mathcal{B} } \chi_E  > \alpha \}|$ were finite,  then $\mathsf M_\mathcal{B}$ satisfied a Tauberian condition with respect to $\alpha$.    The phrase ``Tauberian condition'' was used in a similar manner in \cite{hlp}, but the expression ``Tauberian constant of $M_\mathcal{B}$ with respect to $\alpha$'', referring to the value of the supremum above, was first used in \cite{HP}.

The study of Solyanik estimates is motivated by a variety of classical and modern themes in harmonic analysis. Indeed, although the term \emph{sharp Tauberian constant} is relatively new, the notion is quite standard: for $\lambda\in(1,\infty)$ the function $\phi(\lambda)\coloneqq \C(1/\lambda)$ is known as the \emph{Halo function} of the collection of sets used to define the maximal operator $\M$. Here we use cubes but all these definitions make sense for more general collections of sets. For this point of view and the connection to differentiation properties of bases see for example \cite{Guz}. On the other hand, a Solyanik estimate for the \emph{strong maximal function} has been used in \cite{CLMP} in order to prove a version of Journ\'e's lemma with \emph{small enlargement}. It is actually not hard to see that, in a quite general context, a Solyanik estimate is closely related to a \emph{quantitative C\'ordoba-Fefferman covering lemma}; see \cite{CF}. From this point of view, a Solyanik estimate involves geometric rather than analytical properties of the differentiation basis which is implicit in the definition of $\M$. It comes then as a surprise, as this paper will establish, that Solyanik estimates are also intimately related to the study of $A_\infty$-weights; in particular we will see in \S\ref{s.t2} how the study of Solyanik estimates and sharp Tauberian constants provides a systematic approach for the study of embeddings of $A_\infty$ into $A_p$.

\subsection{Solyanik estimates with respect to \texorpdfstring{$A_\infty$}{Ainf} weights}  The purpose of this paper is the investigation of \eqref{e.main} and its variations under the presence of a weight $w$ in $\R^n$. With more precise definitions and details to follow our first main result is the following:

\begin{theorem}\label{t.Soldoubling} Let $w\in A_\infty$ be a Muckenhoupt weight. Then
	\[
	\Cmu(\alpha)-1 \leq C_{w,n} (1-\alpha )^{c_{w,n}}
	\]
as $\alpha\to 1^-$, where the constants $C_{w,n},c_{w,n}>0$ depend only upon $n$ and $w$.
\end{theorem}
Some remarks are in order.  Firstly, we note that the estimate of the theorem can be reversed in the sense that
\[
\Cmu(\alpha)-1\gtrsim_{w,n} (1-\alpha)^{\tilde c_{w,n}}
\]
for some constant $\tilde c_{w,n}$ depending on $w$ and the dimension. Thus the estimate of the theorem is of the correct form. We do not however pursue the best possible exponent $c_{w,n}$ in the Solyanik estimate above. In fact, already in the case of the Lebesgue measure and the uncentered Hardy-Littlewood maximal operator defined with respect to Euclidean balls, the largest possible value of $c>0$ such that $\C(\alpha)\lesssim_n (1-\alpha)^c$ remains unknown. See \cite{HP} for more details on this issue.

It is natural to consider the estimate stated in Theorem~\ref{t.Soldoubling} for more general measures $\mu$ in the place of $w(x)dx$. That is, given a positive Borel measure $\mu$ in $\R^n$ we can define the maximal operator  $\M_\mu$
\[
 \M_\mu f(x)\coloneqq \sup_{x\in Q}\frac{1}{\mu(Q)} \int_Q |f(y)|d\mu(y);
\]
one then defines $C_\mu$ in the obvious way and asks whether the analogue of \eqref{e.main} is true:
\begin{equation}\label{e.muSol}
\lim_{\alpha\to 1^-} C_\mu(\alpha)=1.
\end{equation}
For example it is easy to see that we have $\C_\mu(\alpha)-1\leq 2\alpha^{-1}(1-\alpha)$ for any locally finite Borel measure $\mu$ on the real line. Using this one trivially gets that \eqref{e.muSol} remains valid whenever $\mu$ is a tensor product of arbitrary locally finite one-dimensional Borel measures. However, \emph{some} restriction will have to be imposed on $\mu$ in order to guarantee the validity of \eqref{e.muSol} in general. This follows easily by considering, for example, a countable collection of cubes all of which contain the origin and such that for every $j\in\{1,2,\ldots\}$ there exists $x_j\in Q_j$ but $x_j \notin Q_k$ for $k\neq j$. If $c_j$ is a sequence of positive numbers such that $\sum_j c_j =+\infty$ but $c_j\to 0$ as $j\to +\infty$ then one easily checks that the operator $\Mmu$ defined with respect to $\mu =\delta_0+\sum_j c_j \delta_{x_j}$ does not satisfy any estimate of the form \eqref{e.main}. In particular $\M_\mu$ is unbounded on all $L^p(\mu)$ for $p<\infty$.

Note that in the case of a doubling measure $\mu$ we know {\it{a priori}} that $M_\mu$ is of weak type $(1,1)$ and thus bounded on $L^p(\mu)$ for all $p\in(1,\infty)$. Furthermore, for any locally finite Borel measure $\mu$  the \emph{centered} Hardy-Littlewood maximal operator $\M_{\mu} ^{\mathsf c}$, defined with respect to $\mu$, is easily seen to satisfy a Solyanik estimate \emph{independently of the measure $\mu$}. Since $\mu$ is doubling the operators $\M_\mu$ and $\M_{\mu} ^{\mathsf c}$ are pointwise equivalent. However, no Solyanik estimate can be deduced from these simple facts. One of the reasons behind this, highlighting the subtlety of the problem, is that Solyanik estimates are very sensitive to constants. Thus there exist examples of pairs of maximal operators which are pointwise comparable, but where one of them satisfies a Solyanik estimate while the other one does not. The examples of this sort that we are aware of exist on the level of the Lebesgue measure and maximal operators $\M_{\mathfrak B}$ defined with respect to rather exotic collections of sets $\mathfrak B$; see \cites{BH,HP} for a related discussion. We are not aware of such examples in the context of the basis of axes-parallel cubes in $\R^n$ and thus it is tempting to conjecture that \eqref{e.muSol} remains valid for $M_\mu$ whenever $\mu$ is a \emph{doubling measure}. However, our current methods do not give a definite answer to this issue which we plan to investigate in a future work.

\subsection{Weighted Solyanik estimates} In this paper we also study a closely related question to the one above, inspired by the rich theory of weighted norm inequalities. We thus consider a non-negative, locally integrable function $w$ on $\R^n$, that is, a weight, and define the \emph{weighted sharp Tauberian constant}
\[
 \Cw(\alpha)\coloneqq \sup_{E:\, 0<w(E)<+\infty} w(E)^{-1} w(\{x\in\R^n:\, \M \chi_E (x)>\alpha\}).
\]
Note here that, in contrast to the definition of $C_{w}(\alpha)$, the operator $\M$ is the Hardy-Littlewood maximal operator \emph{defined with respect to the Lebesgue measure in $\R^n$}. A \emph{weighted Solyanik estimate} is now a statement of the form
\begin{equation}\label{e.main2}
\lim_{\alpha\to 1^-} \Cw(\alpha)=1.
\end{equation}
An immediate consequence of \eqref{e.main2} is that $\M$ satisfies a \emph{weighted Tauberian condition}: there exists \emph{some} $\alpha\in(0,1)$ such that $\Cw(\alpha)<+\infty$. Although apparently weak, a weighted Tauberian condition already encodes the boundedness of $\M$ on $L^p(w)$ for sufficiently large $p$, and thus restricts $w$ to some $A_p$ class of weights. More precisely, as was shown in \cite{hlp}, we have
\[
 A_\infty= \cup_{p\geq 1} A_p = \{w: \Cw(\alpha)<+\infty\quad\text{for some}\quad \alpha\in(0,1)\}.
\]
This immediately tells us that a necessary condition for \eqref{e.main2} is that $w\in A_\infty$. In fact, the converse implication is also true:
\begin{theorem}\label{t.main2} Let $w\in A_\infty$ be a Muckenhoupt weight in $\R^n$. Then
\[
 \Cw(\alpha)-1 \leq C_n \Delta_w ^2 (1-\alpha)^{(c_n[w]_{A_\infty})^{-1}}\quad\text{whenever}\quad 1>\alpha>1-e^{-c_n[w]_{A_\infty}},
\]
where $C_n,c_n>1$ are dimensional constants, and $\Delta_w$ is the doubling constant of $w$. In one dimension we have the same estimate for $\Cw(\alpha)-1$ without the term $\Delta_w ^2$.

Furthermore this estimate is optimal up to the term $\Delta_w ^2$: if $w$ is a non-negative, locally integrable function on $\R^n$ such that
\[
 \Cw(\alpha)-1 \leq B (1-\alpha)^\frac{1}{\beta}\quad\text{whenever}\quad 1 > \alpha>1-e^{-\beta},
\]
for some constants $B,\beta>1$, then $w\in A_\infty$ and $[w]_{A_\infty}\lesssim \beta(1+ \log B) $.
\end{theorem}
The direct part of Theorem~\ref{t.main2} above gives an estimate on the sharp Tauberian constant of $w$ as $\alpha \to 1^{-}$. Since all the constants are so explicit we can use the methods of \cites{hs,hlp} in order to deduce an effective embedding of $A_\infty$ into $A_p$.

\begin{theorem}\label{t.second} Let $w$ be an $A_\infty$ weight in $\R^n$ and define $[w]_{A_\infty}$ to be the Fujii-Wilson constant of $w$. Then $w\in A_p$ for $p> e^{c_n[w]_{A_\infty}}$ with $[w]_{A_p} \leq  e^{e^{c_n[w]_{A_\infty}}} $, where $c_n>1$ is a dimensional constant.
\end{theorem}

We note here that embeddings like the one in Theorem~\ref{t.second} above have been long studied in the theory of weighted norm inequalities and they are typically sharper in one dimension. The most precise result of this kind that we are aware of is contained in \cite{DW}; variations of this embedding depending on different gauges of $A_\infty$ are contained for example in \cites{Kor,Wik}. See also \cites{M,P}. Most of the previously known results use embeddings of the reverse H\"older classes into $A_p$, as for example in \cite{DW}; on the other hand, the embeddings proved in \cite{Kor} use the Hru{\v{s}}{\v{c}}ev constant for $A_\infty$, defined in \cite{Hru}. The result of Theorem~\ref{t.second} is implicit in the literature, at least in dimension $n=1$, as it follows from a combination of the results from \cite{HytP} and \cite{DW}. For higher dimensions, the result of Theorem~\ref{t.second} also follows from a careful reading of \cite{HytP} and \cite{Wik}, but it is not explicitly stated. We include this theorem however, claiming no originality on the result, since our methods are quite different and we think it is worthwhile to highlight this connection.

In order to compare the value of $p\eqsim e^{c_n[w]_{A_\infty}}$ from Theorem~\ref{t.second} to other results in the literature, note for example that $w\in A_\infty$ is essentially equivalent to $w$ satisfying a reverse H\"older inequality
\[
 \bigg(\frac{1}{|Q|}\int_Q w^{1+\frac{1}{c_n[w]_{A_\infty} }} \bigg)^{\frac{1}{1+(c_n[w]_{A_\infty})^{-1}}} \leq 2\frac{1}{|Q|} \int_Q w ;
\]
see for example \cite{HytP} but also Lemma~\ref{l.growth} in the current paper. The exponent $r\coloneqq 1+(c_n[w]_{A_\infty})^{-1}$ in the reverse H\"older inequality above is optimal, up to a dimensional constant. Plugging this information into \cite{DW}*{Theorem 1} gives that, in dimension $n=1$, the optimal value of $p$ such that $w\in A_p$ must satisfy $(p/2)^r -1 =r(p-1)$. Using that $p>2$ and manipulating this equation one obtains $p\gtrsim e^{c [w]_{A_\infty} }$ which matches the embedding into $A_p$ of Theorem~\ref{t.second}.

The rest of this paper is organized as follows. In \S\ref{s.aux} we state and prove some auxiliary asymptotic estimates for enlargements of families of cubes. We also take the chance to recall several known facts about $A_\infty$ weights in a form that will be useful for us later on. In \S\ref{s.t1} we give the proof of Theorem~\ref{t.Soldoubling} based on the lemmas proved in \S\ref{s.aux}. In \S\ref{s.t2} we give the proofs of Theorem~\ref{t.main2} and Theorem~\ref{t.second}. In particular, the proof of Theorem~\ref{t.main2} is contained in Proposition~\ref{p.wT} and Proposition~\ref{p.weighted}.

\section{Notation} We use $C,c>0$ to denote numerical constants whose value may change even in the same line of text. We also write $A\lesssim  B$ if $A\leq  C B$ and $A\eqsim B$ if $A\lesssim B$ and $B\lesssim A$. We state the dependence on the dimension $n$ and the weight $w$ by writing $C_{w,n}$ and $A\lesssim_{w,n} B$ respectively. We typically use the letter $Q$ to denote a cube in $\R^n$ with sides parallel to the coordinate axes and write $cQ$ for the concentric dilation of $Q$ by a factor $c>0$. We write $r_Q$ for the sidelength of $Q$ and $x_Q$ for its geometric center. Finally, weights are non-negative locally integrable functions and $w(Q)\coloneqq \int_Q w$.


\section{An asymptotic estimate for Muckenhoupt weights}\label{s.aux}
Due to the asymptotic nature of Solyanik estimates we many times need to estimate the measure of small dilates of a cube $Q\subset \R^n$ by the measure of the cube itself. To make this more precise let $Q$ be a cube and $\delta\in(0,1)$ be a small parameter. For the Lebesgue measure we then have the trivial estimate
\[
\abs{ (1+\delta)Q \setminus Q } \leq [(1 + \delta)^n - 1]\, |Q| \lesssim_n \delta |Q|.
\]
In order to write down similar estimates for more general measures we recall the notion of a \emph{doubling measure}:

\begin{definition} Let $\mu$ be a (Borel regular) non-negative measure in $\R^n$. We say that $\mu$ is \emph{doubling}  if there exists a constant $\Delta_\mu>0$ such that for every cube $Q\subset \R^n$ we have $\mu(2Q)\leq \Delta_\mu \mu(Q)$. We always assume that $\Delta_\mu>0$ is  the best constant such that the previous inequality holds uniformly for all cubes and we call it the \emph{doubling constant of $\mu$}.
\end{definition}
It is well known that if $\mu$ is a doubling measure in $\R^n$ then there exist constants $c_1,c_2>0$, depending only on $\Delta_\mu$, such that for every cube $Q$ and every $\delta\in(0,1)$ we have
\[
 \mu((1+\delta)Q\setminus Q) \leq c_1 \delta^{c_2} \mu(Q).
 \]
 The previous estimate is in fact equivalent to $\mu$ being doubling. See for example \cite{Kor} for a very nice exposition of this and related results.
%

Asymptotic estimates for single cubes as the ones given above are the first steps to understanding Solyanik estimates. We quickly realize however that we need similar estimates for collections of cubes instead of just single cubes. Thus, given some finite collection of cubes $\{Q_j\}$ and $\delta\in(0,1)$ we will need effective estimates for the measure of $\cup_j (1+\delta)Q_j\setminus \cup_j Q_j$. The main obstruction here is that these cubes might overlap in an arbitrary manner.  A way to deal with this problem is to organize an arbitrary collection of cubes into  \emph{satellite configurations of cubes}.

\begin{definition} Let $\mathcal Q\coloneqq \{Q_j\}_{j=0} ^N$ be a finite collection of cubes in $\R^n$. We say that the collection $\mathcal Q$ is a \emph{satellite configuration of cubes} with center $Q_0$  if  $Q_j\cap Q_0\neq \emptyset$ and $r_{Q_j}\leq r_{Q_0}$ for all $j\in\{1,\ldots,N\}$.
\end{definition}

We mention in passing that a similar (but different) definition has appeared in \cite{FL}. We now give a Lebesgue measure version of the estimate alluded to above.

\begin{lemma}\label{l.lebesgue} Suppose that $\{Q_j\}_{j=1} ^N$ is a collection of cubes in $\R^n$. For all $\delta\in(0,1)$ we have the estimate
	\[
	\Abs{ \bigcup_{j= 0} ^N (1+\delta)Q_j \setminus \bigcup_{j= 0} ^N Q_j } \lesssim_n \delta \Abs{\bigcup_{j=0} ^N Q_j}.
	\]
In particular, if $\{Q_j\}_j$ is a satellite configuration of cubes with center $Q_0$ then the right hand side in the estimate above is $\lesssim_n\delta |Q_0|$.
\end{lemma}

\begin{proof} For $x=(x_1,\ldots,x_n)\in\R^n$ we use the standard $\ell^\infty$-norm defined as $\|x\|_{\ell^\infty}\coloneqq \sup_{1\leq j \leq n}|x_j|$. Furthermore, in order to avoid confusion, we write $(1+\delta)_j S$ for the dilation of a set $S$ with respect to the center of $Q_j$ by a factor $(1+\delta)$. Let us assume that the cubes $Q_j$ are ordered so that their sidelengths are decreasing and set $E_0\coloneqq Q_0$ and $E_j\coloneqq Q_j \setminus \cup_{k<j}Q_k$ for $j\geq 1$. Obviously the sets $E_j$ are pairwise disjoint and $\cup_j Q_j=\cup_j E_j$. Furthermore we claim that for all $\delta\in(0,1)$ and all $k\leq N$ we have the identity
	\begin{equation}\label{e.increm}
	\bigcup_{j=0} ^k (1+\delta)_j Q_j =\bigcup_{j=0} ^k (1+\delta)_j E_j.
	\end{equation}
We prove the claim by induction on $k$. For $k=0$ the claim is trivial so we assume that it is true for $k\geq 0$. It will suffice to show the inclusion
\[
\bigcup_{j=0} ^{k+1} (1+\delta)_j Q_j \subseteq \bigcup_{j=0} ^{k+1} (1+\delta)_j E_j,
\]
the opposite inclusion being trivial. By the inductive hypothesis we have
\[
\bigcup_{j=0} ^{k+1} (1+\delta)_j Q_j =\bigcup_{j=0} ^{k+1} (1+\delta)_j E_j \cup (1+\delta)_{k+1} \big(Q_{k+1} \cap \bigcup_{j<k+1}Q_j \big).
\]
We will show that
\[
 (1+\delta)_{k+1}\big(Q_{k+1} \cap \bigcup_{j<k+1}Q_j\big) \subseteq \bigcup_{j<k+1} (1+\delta)_jE_j.
\]

Let us assume that $x\in (1+\delta)_{k+1}(Q_{k+1} \cap \cup_{j<k+1}Q_j) \neq \emptyset$. Then there exists some $J< k+1$ such that $x\in (1+\delta)_{k+1}(Q_{k+1}\cap Q_J)$. This means that there exists some $z\in Q_{k+1}\cap Q_J$ such that $x=x_{Q_{k+1}}+(1+\delta)(z-x_{Q_{k+1}})$. From this we can write
\[
\begin{split}
\|x-x_{Q_J}\|_{\ell^\infty}&= \|\delta(z -  x_{Q_{k+1}})+(z-x_{Q_J})  \|_{\ell^\infty}  \leq \delta \|z-x_{Q_{k+1}}\|_{\ell^\infty}+ \|z-x_{Q_J}\|_{\ell^\infty}
\\
&\leq \delta \frac{r_{Q_{k+1}}}{2}+\frac{r_{Q_J}}{2} \leq \frac{r_{Q_J}}{2} ( 1+\delta),
\end{split}
\]
the last inequality following since $J<k+1$ and thus $r_{Q_J}\geq r_{Q_{k+1}}$. However this means that $x\in(1+\delta)_JQ_J$. Thus $x\in \cup_{j<k+1} (1+\delta)_j Q_j = \cup_{j<k+1}(1+\delta)_j E_j$ by the inductive hypothesis. This completes the inductive proof and shows that \eqref{e.increm} holds.

Having \eqref{e.increm} at our disposal the rest is routine. Indeed we have
\[
\Abs{\bigcup_{j=0} ^N (1+\delta)_jQ_j } = \Abs{ \bigcup_{j=0} ^N (1+\delta)_jE_j } \leq (1+\delta)^n \sum_{j=0} ^N |E_j|\leq (1+c_n\delta )\Abs{\bigcup_{j=0} ^N Q_j}.
\]
This proves the main claim of the lemma. If $\{Q_j\}_{j=0} ^N$ is a satellite configuration with center $Q_0$ it is immediate that $|\cup_j Q_j|\lesssim_n |Q_0|$.
\end{proof}

\begin{remark}\label{r.balls} It is worth noticing that Lemma~\ref{l.lebesgue} above remains true whenever we have some satellite configuration of \emph{metric balls} in $\R^n$. Indeed, the only thing needed for the proof is identity \eqref{e.increm} which in turn is true whenever our sets $Q_j$ are balls with respect to the same metric in $\R^n$.\qed
\end{remark}
 We now desire to prove a weighted analog of Lemma~\ref{e.increm} above. For this we need to divert a bit and recall some known facts about the class of Muckenhoupt weights $A_\infty$.

\subsection{The class of Muckenhoupt weights \texorpdfstring{$A_\infty$}{Ainf}}  Recall here that $A_\infty$ can be defined as the class of non-negative, locally integrable functions $w$ in $\R^n$ such that
\[
[w]_{A_\infty}\coloneqq \sup_Q \frac{1}{w(Q)}\int_Q \M(w\chi_Q) <+\infty;
\]
the supremum is taken with respect to all cubes in $\R^n$. An equivalent description of $A_\infty$ is
\[
A_\infty=\bigcup_{p\geq 1}A_p
\]
where $A_p$ denotes all the weights $w$ for which
\[
[w]_{A_p}\coloneqq \sup_Q  \bigg(\frac{1}{|Q|}\int_Q w\bigg) \big(\frac{1}{|Q|}\int_Q w^{-\frac{1}{p-1}}\big)^{p-1}<+\infty,
\]
and the supremum is taken over all cubes in $\R^n$. The definition of the class $A_\infty$ by means of the quantity $[w]_{A_\infty}$ goes back to Fujii, \cite{Fu}, and Wilson, \cites{W1,W2}, and has recently been used in order to prove sharp quantitative weighted norm inequalities for maximal functions and singular integrals. See for example \cites{HytP,HytPR,LM}. Further properties of $A_\infty$ weights and equivalent definitions are discussed in many places as for example in \cite{GaRu}. See however \cite{DMRO} for a more up to date discussion of the equivalent definitions of $A_\infty$.

The lemma below encodes some of the deepest properties of $A_\infty$ weights. In particular, the proof of the lemma crucially depends on the reverse H\"older inequality for $A_\infty$ weights with sharp exponent. The characterization of the quantity $[w]_{A_\infty}$ given below is implicit in \cite{HytP}*{p.24--26}.

\begin{lemma}\label{l.growth} Let $w$ be a non-negative, locally integrable function on $\R^n$. The following hold:
\begin{itemize}
\item[(i)] If $w\in A_\infty$ and $[w]_{A_\infty}$ denotes the Fujii-Wilson constant as defined above, then for any cube $Q\subset \R^n$ and any measurable set $S\subseteq Q$ we have
\[
\frac{w(S)}{w(Q)}\leq 2 \big(\frac{|S|}{|Q|}\big)^{(c_n[w]_{A_\infty})^{-1}}
\]
where $c_n>1$ is a dimensional constant.
\item[(ii)] Conversely, if there exist constants $c_1,c_2> 1$ such that for every cube $Q\subset \R^n$ and every measurable $S\subseteq Q$ we have
\[
\frac{w(S)}{w(Q)}\leq c_1 \big(\frac{|S|}{|Q|}\big)^\frac{1}{c_2}
\]
then $w\in A_\infty$ and $[w]_{A_\infty}\lesssim c_2(1+\log c_1)$.
\end{itemize}
\end{lemma}

\begin{proof}
If $w\in A_\infty$ we have from \cite{HytP}*{Theorem 2.3} that there exists a dimensional constant $c_n>1$ such that for every cube $Q\subset \R^n$ we have the reverse H\"older inequality
\[
\left(\frac{1}{|Q|}\int_Q w^{1+(c_n[w]_{A_\infty})^{-1}}\right)^{\frac{1}{1+(c_n[w]_{A_\infty})^{-1}}} \leq 2 \frac{1}{|Q|}\int_Q w\;.
\]	
The conclusion in (i) easily follows from this via H\"older's inequality.

In order to prove (ii) let us fix a cube $Q$ and show that $w$ satisfies a reverse H\"older inequality on $Q$. Without loss of generality we can assume that $w(Q)/|Q|=1$. Observe also that the hypothesis implies that
\[
\frac{w(S)}{w(Q)}\leq \big(\frac{|S|}{|Q|}\big)^\frac{1}{2c_2}\quad\text{whenever}\quad \frac{|S|}{|Q|}\leq e^{-2c_2(1+\log c_1) }.
\]
However when $|S|/|Q|>e^{-2c_2(1+\log c_1)}$ it is trivial that
\[
\frac{w(S)}{w(Q)}\leq 1 \leq e \big(\frac{|S|}{|Q|}\big)^\frac{1}{2c_2(1+ \log c_1)}\;.
\]
Combining these estimates we get that in every case
\[
\frac{w(S)}{w(Q)}\leq e \big(\frac{|S|}{|Q|}\big)^\frac{1}{2c_2(1+\log c_1)}
\]
for all cubes $Q\subseteq \R^n$ and all measurable $S\subseteq Q$.  Let us set $\eta\coloneqq 	2c_2(1+\log c_1)$ for the rest of the proof in order to simplify the notation and for $\lambda>0$ we write $E_\lambda \coloneqq\{x\in Q:\, w(x)>\lambda\}$. Observe then that
\[
\frac{|E_\lambda|}{|Q|}\leq\frac{1}{\lambda}\frac{w(E_\lambda)}{w(Q)} \frac{w(Q)}{|Q|}\leq e \frac{1}{\lambda} \big(\frac{|E_\lambda|}{|Q|}\big)^\frac{1}{\eta}\quad\text{and thus}\quad \frac{|E_\lambda|}{|Q|}\leq e^{\eta'}\frac{1}{\lambda^{\eta'}}
\]
where $\eta'$ is the dual exponent of $\eta$. Thus
\[
\frac{w(E_\lambda)}{w(Q)}\leq e \big(\frac{|E_\lambda|}{|Q|}\big)^\frac{1}{\eta}\leq e^{1+\frac{\eta'}{\eta}} \lambda^{-\frac{\eta'}{\eta}}\lesssim \lambda^{-\frac{\eta'}{\eta}}
\]
since $\eta>2$. For $\epsilon>0$ we now write
\[
\begin{split}
\int_Q w^{1+\epsilon}&=\int_0 ^\infty w(E_\lambda)\epsilon  \lambda^{\epsilon-1}d\lambda\lesssim w(Q)+\epsilon w(Q)\int_1 ^\infty  \lambda^{\epsilon-1-\frac{\eta'}{\eta}}d\lambda
\\
&=w(Q)\big(1+\frac{\epsilon}{\frac{\eta'}{\eta}-\epsilon}\big)
\end{split}
\]
as long a $\epsilon<\eta'/\eta$. Now we readily see that for $\epsilon\leq \eta' /2\eta <1/\eta$ the weight $w$ satisfies for every cube $Q$ the reverse H\"older inequality
\[
\bigg(\frac{1}{|Q|}\int w^{1+\epsilon}\bigg)^\frac{1}{1+\epsilon}\lesssim \frac{1}{|Q|}\int_Q  w.
\]
By Theorem~\cite{HytP}*{Theorem 2.3 (b)} we now get that $[w]_{A_\infty}\lesssim_n (1+1/\eta)' \eqsim \eta= c_2(1+\log c_1)$ as we wanted.
\end{proof}
The following standard estimates for $A_\infty$ weights follow immediately from the previous lemma. We omit the simple proof.

\begin{corollary}\label{c.growth}Let $w\in A_\infty$.
\begin{itemize}
	\item [(i)] The measure $d\mu(x)\eqqcolon w(x)dx$ is doubling with doubling constant
	\[
	\Delta_w \lesssim \exp \big[\log 2 / \log(1+C_n e^{-c_n[w]_{A_\infty}})\big]\eqsim e^{e^{c_n[w]_{A_\infty}}} ,
	\]
	where $c_n>0$ is a dimensional constant.
	\item[(ii)] For every  two cubes $Q_1\subseteq Q_2$ with sidelengths $r_1,r_2$, respectively, we have the estimate
	\[
	 \frac{ w(Q_1) }{ w(Q_2) } \gtrsim_{w,n}  \big(  \frac{r_1}{r_2} \big) ^ {\gamma_{w,n} } ;
	\]
	the constants depend only upon $w$ and the dimension.
	\end{itemize}
\end{corollary}
We now state a weighted version of Lemma~\ref{l.lebesgue}.
\begin{lemma}\label{l.doubling}Let $w\in A_\infty $ be a Muckenhoupt weight in $\R^n$ and $\mathcal Q=\{Q_j\}_{j=0} ^N$ be a satellite configuration of cubes with center $Q_0$. Then for all $\delta\in(0,1)$ we have
	\[
w\big( \bigcup_{j= 0} ^N (1+\delta)Q_j \setminus \bigcup_{j= 0} ^N Q_j \big) \lesssim_{n} \Delta_w ^2 \delta ^{(c_n[w]_{A_\infty})^{-1}}  w(Q_0)
	\]
where $c_n>1$ is a dimensional constant and $\Delta_w$ is the doubling constant of $w$.
\end{lemma}

\begin{proof} Since $\{Q_j\}_{j=0} ^N$ is a satellite configuration of cubes it is immediate that $\cup_j Q_j \subseteq 3Q_0$ and thus $\cup_j(1+\delta)Q_j \subseteq 4Q_0$. Setting $E\coloneqq \cup_{j= 0} ^N (1+\delta)Q_j \setminus \cup_{j= 0} ^N Q_j \subset 4Q_0$ and using Lemma~\ref{l.lebesgue} we have that $|E|\lesssim_n \delta |4Q_0|$. Since $w\in A_\infty$ we can conclude that
\[
\frac{w(E)}{w(4Q_0)} \leq 2 \Big(\frac{|E|}{|4Q_0|}\Big)^{  (c _n[w]_{A_\infty})^{-1}}\lesssim_n \delta^ {  (c _n[w]_{A_\infty})^{-1}}.
\]
However, since $w\in A_\infty$ we have that $w$ is doubling which gives $w(E)\lesssim_{n} \Delta_w ^2  \delta^{  (c _n[w]_{A_\infty})^{-1}} w(Q_0)$ as we wanted.
\end{proof}

The heart of the matter is the following asymptotic estimate controlling the $w$-measure of unions of $(1+\delta)$-enlarged cubes by the $w$-measure of the union of the cubes themselves.

\begin{lemma}\label{l.collect}Let $w\in A_\infty$ be a Muckenhoupt weight on $\R^n$ and let $\{Q_j\}_{j=0} ^N$ be a finite collection of cubes. Then
	\[
	w\big(\bigcup_j(1+\delta)Q_j \setminus \bigcup_j Q_j\big)\lesssim_{w,n} \delta^{  (c _n[w]_{A_\infty})^{-1}} w\big(\bigcup_j Q_j\big);
	\]
the  constant $c_n>1$ depends only upon $n$.
\end{lemma}

\begin{proof} Given the finite collection of cubes $\mathcal Q\coloneqq \{Q_j\}_j$ we assume that they are ordered according to decreasing sidelengths. By the Vitali covering lemma there exists a disjoint subcollection $\mathcal R\coloneqq \{R_j\}_j \subseteq \mathcal Q$ such that if $Q\in \mathcal Q\setminus \mathcal R$ were not selected then there exists $R\in\mathcal R$ with $R\cap Q\neq \emptyset$ and $r_R\geq r_Q$. Of course we have $\cup_j Q_j \subseteq \cup_j 3R_j$.

For each $R\in\mathcal R$ we let $\mathcal Q_R$ be the collection of cubes $Q\in\mathcal Q$ such that $Q\cap R\neq\emptyset$ and $r_Q\leq r_R$. By the previous comments we have the identity
	\[
	\bigcup_{Q\in\mathcal Q}Q = \bigcup_{R\in\mathcal R} \bigcup_{Q\in\mathcal Q_R} Q\;.
	\]
Note that each collection $\mathcal Q_R$ is a satellite configuration with center $R$ and that $R\in\mathcal Q_R$. Thus, Lemma~\ref{l.doubling} implies that
\[
w\big(\bigcup_{Q\in\mathcal Q_R}(1+\delta)Q\setminus \bigcup_{Q\in\mathcal Q_R }Q\big)\lesssim_{w,n} \delta^{(c_n [w]_{A_\infty})^{-1}}w(R).
\]
It now follows that
\[
\begin{split}
w\big(\bigcup_{Q\in\mathcal Q}(1+\delta)Q\setminus \bigcup_{Q\in\mathcal Q }Q\big)&\leq \sum_{R\in\mathcal R}w\big(\bigcup_{Q\in\mathcal Q_R}(1+\delta)Q\setminus \bigcup_{Q\in\mathcal Q_R }Q\big)
\\
&\lesssim_{w,n} \delta^{(c_n [w]_{A_\infty})^{-1}} \sum_{R\in\mathcal R} w(R)= \delta^ {(c_n[w]_{A_\infty})^{-1}} w(\cup_{R\in\mathcal R} R)
\\
&  \lesssim \delta^{(c_n [w]_{A_\infty})^{-1}}w(\cup_j Q_j),
\end{split}
\]
since all the $R\in\mathcal R$ are pairwise disjoint and contained in $\cup_j Q_j$. This proves the lemma.
\end{proof}

\section{A Solyanik estimate for Muckenhoupt weights}\label{s.t1} In this section we turn to our main task of proving a Solyanik estimate for the uncentered maximal function $\Mmu$ defined with respect to cubes and a Muckenhoupt weight $w$ in $\R^n$. The following covering argument is a weighted variation of the argument used in \cite{HP}*{Theorem 3}.

\begin{lemma}\label{l.covering} Let $\{Q_j\}_{j=1} ^N$ be a finite collection of cubes in $\R^n$ and $\xi\in(0,1)$, and let $w\in A_\infty$ be a Muckenhoupt weight in $\R^n$. There exists a subcollection $\{\tilde Q_k\}_{k=1} ^M\subseteq \{Q_j\}_{j=1} ^N$ such that
	\begin{itemize}
		\item[(i)] We have the estimate
		\[
	w\big(\bigcup_{j=1} ^N Q_j\big) \leq (1+C_{w,n} \xi^{c_{w,n}}) w\big(\bigcup_{j=1} ^M \tilde Q_j\big),
	\]
where the constants $C_{w,n},c_{w,n}>1$ depend only on the dimension $n$ and the weight $w$.	
	\item[(ii)] For every $j\geq 2$ we have  $w(  \tilde Q_j\setminus\cup_{k<j}\tilde Q_k ) >\xi w(\tilde Q_j). $
		\end{itemize}
\end{lemma}

\begin{proof}The proof is a modified C\'ordoba-Fefferman selection algorithm inspired by the one in \cite{CF}. We begin by ordering our cubes $\{Q_j\}$ so that $|Q_1|\geq |Q_2|\geq \cdots\geq |Q_N|$ and we select $\tilde Q_1 \coloneqq Q_1$. Now assuming $\tilde Q_1,\ldots,\tilde Q_j \eqqcolon Q_J $ have been selected we choose $\tilde Q_{j+1}$ to be the first cube $Q\in\{Q_{J+1},\ldots,Q_N\}$ that satisfies
	\[
	w(Q\cap \cup_{k\leq j}\tilde Q_j) \leq (1-\xi) w(Q).
	\]
If no such cube can be selected the selection algorithm terminates. Now if $Q$ is one of the cubes not selected and $x\in Q$ we consider a cube $Q_x\subseteq Q$ such that $x\in Q_x$ and $w(Q_x)=\xi w(Q)$. This is in fact possible since any $A_\infty$-weight does not charge boundaries of cubes. Since the cube $Q$ was not selected we have that
\[
w(Q\cap \cup_{j<j_Q}) \tilde Q_j ) >(1-\xi)w(Q)
\]
for some $j_Q\leq M$. This implies that the cube $Q_x$ must intersect one of the cubes $\tilde Q_j$ selected before, and thus of larger sidelength than $Q$. Furthermore, we can use Corollary~\ref{c.growth} (ii) in order to obtain
\[
\xi = \frac{w(Q_x)}{w(Q)} \gtrsim_{w,n}  \big( \frac {r_{Q_x}}{r_Q} \big)^{\gamma_{w,n}}\;.
\]
The last estimate together with the triangle inequality and the discussion above show the inclusion
\[
\bigcup_{j=1} ^N Q_j\subseteq \bigcup_{j=1} ^M  (1+  C_{w,n} \xi^{c_{w,n} })\tilde Q_j
\]
for some constants $ C_{w,n},c_{w,n}$ depending only on $n$ and $w$. Part (i) of the lemma now follows from Lemma~\ref{l.collect} while part (ii) is automatically satisfied because of the selection algorithm.
\end{proof}
\begin{proof}[Proof of Theorem~\ref{t.Soldoubling}] Let $E$ be a measurable set in $\R^n$ and consider any $K\subseteq E_\alpha\coloneqq\{x\in\R^n:\, \Mmu(\chi_E)(x)>\alpha\}$, where $K$ is compact. There exists a finite collection of cubes $\{Q_j\}_j$ such that $K\subseteq \cup_j Q_j$ and for each $j$ we have $w(E\cap Q_j) >\alpha w(Q_j)$. For $\xi\in(0,1)$, to be chosen momentarily, we let $\{\tilde Q_j\}_j$ be the subcollection of $\{Q_j\}_j$ provided by Lemma~\ref{l.covering}; we have
\[
 w(K)\leq w(\cup_j Q_j)\leq	(1+C_{w,n}\xi^{c_{w,n}}) w(\cup_j \tilde Q_j).
\]
Now letting $\tilde E_j\coloneqq \tilde Q_j \setminus \cup_{k< j} \tilde Q_j $ for $j\geq 2$ and $\tilde E_1\coloneqq Q_1$, part (ii) of Lemma~\ref{l.covering} implies that
\[
w( E\cap \tilde E_j) > \big(1-\frac{1-\alpha}{\xi}\big)w(\tilde E_j).
\]
Thus for $\xi \in (1-\alpha,1)$ we conclude
\[
w(\cup_j \tilde Q_j) =\sum_j w(\tilde E_j) \leq \frac{\xi}{\xi-(1-\alpha)}w(E).
\]
Combining the previous estimates yields
\[
w(K)\leq (1+C_{w,n}\xi^{c_{w,n}})\frac{\xi}{\xi-(1-\alpha)}w(E).
\]
Choosing $\xi \coloneqq (1-\alpha)^{\frac{1}{c_{w,n}+1}}$ gives the theorem.
\end{proof}

\section{Weighted Solyanik estimates and allied issues}\label{s.t2}
In this concluding section we make a comparison between Solyanik estimates and corresponding results in the literature of weighted norm inequalities for maximal functions and related issues. We also discuss \emph{weighted Solyanik estimates} for $A_\infty$ weights. The difference is that now there are two measures involved in such an estimate: the Hardy-Littlewood maximal operator $\M$ is defined with respect to the Lebesgue measure while the measure in the ambient space is $w\in A_\infty$.
\subsection{Pointwise coverings and Solyanik estimates with respect to measures}\label{s.easy} Let $\M_\mu ^{\mathsf{c}}$ denote the centered Hardy-Littlewood maximal operator defined with respect to a locally finite non-negative Borel measure $\mu$:
\[
\M_\mu ^{\mathsf c}f(x)\coloneqq \sup_{r>0}\frac{1}{\mu(Q(x,r))}\int_{Q(x,r)} |f(y)|d\mu(y);
\]
here $Q(x,r)$ is a cube in $\R^n$ with center $x$ and sidelength $2r$. It is well known that this operator is bounded on $L^p(\mu)$ independently of the measure $\mu$ and this is a direct consequence of the Besicovitch covering lemma. The same is true for the dyadic maximal operator $\M_\mu ^{\mathsf{d}}$ defined with respect to $\mu$ and the one-dimensional non-centered maximal operator $\M_\mu ^{1}$ defined with respect to $\mu$. Underlying all these cases is a sharp covering lemma; we already mentioned Besicovitch for the centered operator while for the dyadic maximal function this is a consequence of the ultrametric structure of dyadic cubes; in one dimension we have the powerful covering lemma, specific to the topology of the real line, which says that any collection of intervals can be covered by a subcollection of (pointwise) overlap at most $2$.

It is an easy guess that for these operators we should have Solyanik estimates independently of the measure $\mu$ involved in their definition. This turns out to be true and the proof of this fact is a trivial modification of the proof of the corresponding weak type $(1,1)$ inequalities.

For example, for $\M_\mu ^1$ we have that any compact set $K\subseteq \{x\in \R:\, \M_\mu ^{1}(\chi_E)>\alpha\}$ can be covered by a finite sequence of intervals $I_j$ such that $\mu(I_j\cap E)>\alpha \mu(I_j)$ and $\sum_j \chi_{I_j}\leq 2$; see for example~\cite{Gar}*{p. 24}. It follows that
\[
\mu(K) \leq \mu(E)+\sum_{j}\mu(E^\mathsf{c}\cap I_j)\leq \mu(E)+\frac{1-\alpha}{\alpha}\sum_j \mu(I_j\cap E)\leq (1+2\frac{1-\alpha}{\alpha})\mu(E)
\]
and thus, in this case $\C_\mu(\alpha)\leq \frac{2-\alpha}{\alpha}$. It is also easy to check that this bound is best possible. Completely analogous arguments work for the other operators discussed in this paragraph and give estimates of the type $\C_\mu(\alpha)-1\eqsim_n (1-\alpha)$, independently of the measure $\mu$.

\subsection{Weighted Tauberian constants, Solyanik estimates and \texorpdfstring{$A_\infty$}{Ainf}} Things get more interesting if we consider a non-negative locally integrable function $w$, that is, a weight, and define the weighted Tauberian constant
\[
\Cw(\alpha)\coloneqq \sup_{E:\,0<|E|<\infty}w(E)^{-1} w(\{x\in\R^n:\, \M \chi_E(x)>\alpha\}).
\]
Note here that $\M$ is the usual non-centered Hardy-Littlewood maximal function defined with respect to the Lebesgue measure. It is natural to ask under what conditions on $w$ we have a Solyanik estimate
\begin{equation}\label{e.weightSol}
\lim_{\alpha\to 1^- }\Cw(\alpha)=1.
\end{equation}
By the results in \cite{hlp} we know that $w\in A_\infty$ if and only if there exists $\alpha\in(0,1)$ such that $\Cw(\alpha)<+\infty$. This equivalence was proved in \cite{hlp} by means of general arguments relating weighted Tauberian constants with the $L^p(w)$-boundedness of $\M$ for large $p$. Thus a necessary condition for \eqref{e.weightSol} is that $w\in A_\infty$. Theorem~\ref{t.main2} claims that the opposite implication is also true, namely \eqref{e.weightSol} holds whenever $w\in A_\infty$. The remainder of this paper is devoted to the proof of this theorem. We also take the chance to state a quantitative version. The essence of this quantification is that the $A_\infty$-constant of a weight $w$ is, roughly speaking, equivalent to the  smallest number $c>1$ such that the following asymptotic estimate holds:
\[
 \Cw(\alpha)-1 \lesssim (1-\alpha)^\frac{1}{c}\quad\text{as}\quad \alpha\to 1^-.
\]
We make this precise in the following proposition which proves the optimality part of Theorem~\ref{t.main2}.
\begin{proposition}\label{p.wT} Suppose that $w$ is a non-negative, locally integrable function in $\R^n$ such that the following Solyanik estimate holds:
\[
 \Cw(\alpha)-1 \leq B (1-\alpha)^\frac{1}{\beta}\quad\text{whenever}\quad \alpha > 1-e^{-\beta},
\]
for some numerical constants $B,\beta\geq 1$. Then $w\in A_\infty$ and $[w]_{A_\infty}\lesssim_n \beta (1+\log B)$.
\end{proposition}
\begin{proof} For a cube $Q\subset \R^n$ and a measurable set $A\subset Q$ with $|A|/|Q|>\alpha$ we have $w(Q) \leq \Cw(\alpha) w(A)$ and taking complements we get
\begin{equation}\label{e.finite}
 w( S) \leq  \frac{\Cw(\alpha)-1}{\Cw(\alpha)} w(Q) \quad\text{whenever}\quad |S|/|Q|<1-\alpha <e^{-\beta}.
\end{equation}
This already implies that $w\in A_\infty$; see Remark~\ref{r.ainf}. To get an estimate for $[w]_{A_\infty}$ observe that since $\Cw(\alpha)\geq 1$ we have
\[
\frac{w(S)}{w(Q)}\leq  \Cw(\alpha)-1 \leq B(1-\alpha)^\frac{1}{\beta }.
\]
Letting $\alpha \to 1-|S|/|Q|$ we get that for all $S\subset Q$ with $|S|/|Q|<e^{-\beta}$ we have
\[
 \frac{w(S)}{w(Q)} \leq B  \big( \frac{ |S|}{|Q|}\big)^\frac{1}{\beta } .
\]
When $|S|/|Q|>e^{-\beta}$ we trivially have
\[
 \frac{w(S)}{w(Q)} \leq e  \big( \frac{ |S|}{|Q|}\big)^\frac{1}{\beta } .
\]
By Lemma~\ref{l.growth} it follows that $[w]_{A_\infty}\lesssim \beta (1+\log B)$.
\end{proof}
\begin{remark}\label{r.ainf} Proposition~\ref{p.wT} above proves that a Solyanik estimate implies that $w\in A_\infty$ and quantifies this implication in terms of the involved constants. Observe also that by \eqref{e.finite} we see that if $\Cw(\alpha)<+\infty$ then there exist constants $\xi,\eta<1$ such that $|S|/|Q|<\xi\Rightarrow w(S)/w(Q)<\eta$. This condition is a well known equivalent characterization of $A_\infty$; see for example \cite{GaRu}*{Corollary IV.2.13}. On the other hand, if we assume that $w\in A_\infty$ then $w\in A_p$ for some $p\in[1,\infty)$ and thus $\M:L^p(w)\to L^{p,\infty}(w)$, where $\M$ is the Hardy-Littlewood maximal operator on $\R^n$. This immediately implies that
\[
 w(\{x\in\R^n:\, \M\chi_E(x)>\alpha\})\lesssim_{w,n,p} {\alpha^{-p}} w(E),\quad \alpha\in(0,1),
\]
for all measurable sets $E\subseteq\R^n$. Thus we get a trivial proof of the fact that
\[
 w \in A_\infty \Leftrightarrow \exists \alpha\in(0,1)\quad\text{such that}\quad \Cw(\alpha)<+\infty.
\]
However, the proof of the same fact from \cite{hlp}, which is substantially more involved, gives more precise quantitative information; see also the proof of Theorem~\ref{t.second}.\qed
\end{remark}

We now move to the direct implication of Theorem~\ref{t.main2} which is the content of the following proposition:
 \begin{proposition}\label{p.weighted} Let $w\in A_\infty$ and $\M$ denote the Hardy-Littlewood maximal operator in $\R^n$, defined with respect to cubes. We have the Solyanik estimate
 	\[
 	\Cw(\alpha)-1\lesssim_n \Delta_w ^2 (1-\alpha)^{(c_n[w]_{A_\infty}) ^{-1}}\quad\text{whenever}\quad \alpha>1-e^{-c_n[w]_{A_\infty}}.
 	\]
Here $\Delta_w$ is the doubling constant of $w$, and $c_n$ and the implied constant depend only upon the dimension $n$.
 \end{proposition}
\begin{proof} As usual, let us consider a compact set $K\subseteq E_\alpha\coloneqq \{ x\in\R^n:\M\chi_E>\alpha\}$ and so $K\subseteq \cup_j Q_j$ for some finite collection of cubes $\mathcal Q\coloneqq \{Q_j\}_{j=1} ^N$ with $|Q_j \cap E|>\alpha |Q_j|$ for all $j$. We assume that $|Q_1|\geq \cdots\geq |Q_N|$ and for $\delta\in(0,1)$ to be chosen later we choose $\{\tilde Q_j\}_{j=1}$ to be a C\'ordoba-Fefferman subcollection at level $1-\delta$. Namely, we set $\tilde Q_1\coloneqq Q_1$ and, assuming $\tilde Q_1, \ldots,\tilde  Q_j\eqqcolon Q_J$ have been selected, we set $\tilde Q_{j+1}$ to be the first cube among the cubes $Q\in \{Q_{J+1},\ldots,Q_N\}$ that satisfy
\[
 \Abs{Q\cap \bigcup_{k\leq j} \tilde Q_j} \leq (1-\delta) |Q|.
\]
Just as in the proof of Lemma~\ref{l.covering} we have
\[
 \bigcup_j Q_j \subseteq \bigcup_j (1+c_n\delta^\frac{1}{n})\tilde Q_j.
\]
The proof is based on the following basic estimate
\[
 w(E_\alpha)\leq w(E)+w(\cup_j (1+c_n\delta^\frac{1}{n})\tilde Q_j \setminus \cup_j\tilde Q_j )+ w(\cup_j \tilde Q_j \cap E^{\mathsf{c}} )\eqqcolon w(E)+I+II.
\]
Now let $\mathcal R$ be the Vitali subcollection of $\{\tilde Q_j\}_j$	 so that $\cup_j \tilde Q_j \subseteq \cup_{R\in\mathcal R} 3R$, the cubes in $\mathcal R$ are pairwise disjoint and each $\tilde Q_j$ intersects some $R\in\mathcal R$ of larger sidelength. We organize the collection $\{\tilde Q_j\}_j$ into satellite configurations; for $R\in\mathcal R$ we let $\mathcal Q_R$ be the cubes from $\{\tilde Q_j\}_j$ that intersect $R$ and have smaller sidelength than the sidelength of $R$. Thus
\[
\bigcup_j \tilde Q_j = \bigcup_{R\in\mathcal R} \bigcup_{Q\in\mathcal Q _R} Q
\]
and each $Q_R$ is a satellite configuration with center $R\in \mathcal R$.

In order to estimate $I$ observe that
\[
 I\leq\sum_{R\in\mathcal R} w\big(\bigcup_{Q\in\mathcal Q_R} (1+c_n\delta^\frac{1}{n}) Q\setminus \bigcup_{Q\in\mathcal Q_R} Q\big)\lesssim_n \Delta_w ^2 \delta^{(c_n[w]_{A_\infty})^{-1}} \sum_{R\in\mathcal R}w(R)
\]
by Lemma~\ref{l.doubling}. Tracing back the definition of any cube $R$ we remember that $|R\cap E^\mathsf{c}|\leq (1-\alpha)|R|$ and thus, by the $A_\infty$ hypothesis on $w$, we get
\[
 w(R\cap E^{\mathsf c}) \leq 2 (1-\alpha)^{ (c_n [w]_{A_\infty} )^ {-1} } w(R)
\]
so that
\begin{equation}\label{e.RcapE}
w(R) \leq \big(1-2(1-\alpha)^{(c_n[w]_{A_\infty})^{-1}}\big)^{-1} w(R\cap E)\quad\text{whenever}\quad \alpha >1-e^{-C_n [w]_{A_\infty}}.
\end{equation}
Here $C_n>1$ is some large dimensional constant. Plugging into the estimate for $I$ we can conclude
\[
 I\lesssim_n \frac{\Delta_w ^2 \delta^{(c_n[w]_{A_\infty})^{-1}} }{ 1-2(1-\alpha)^{(c_n[w]_{A_\infty})^{-1}}} w( E)\quad\text{whenever}\quad \alpha >1-e^{-C_n [w]_{A_\infty}}
\]
since all the $R\in\mathcal R$ are pairwise disjoint.

We now move to the estimate for $II$ using again the organization of the collection $\{\tilde Q_j\}_j$ into satellite configurations. We have
\[
 II \leq \sum_{R\in\mathcal R} w\big(\bigcup_{Q\in\mathcal Q_R} Q \cap E^{\mathsf{c}}\big)\eqqcolon \sum_{R\in\mathcal R} w(A_R).
\]
For each $R$ observe that $A_R=\bigcup_{Q\in\mathcal Q_R} Q \cap E^{\mathsf{c}}\subseteq 3R$ since $\mathcal Q_R$ is a satellite configuration with center $R$ and let us write out the cubes in $\mathcal Q_R$ as $\mathcal Q_R\eqqcolon\{S_0,S_1,\ldots,S_K\}$ where $|S_0|\geq \cdots\geq |S_K|$. We define the increments $E_k$ as usual by setting $E_0\coloneqq S_0=R$ and $E_k\coloneqq S_k\setminus \cup_{\ell<k} S_\ell$ for $k\geq 1$. Observe that since $\mathcal Q_R \subset \{\tilde Q_j\}_j$ we have that $|E_k|>\delta |S_k|$ and $|S_k \cap E|>\alpha|S_k|$ for every $k$. Thus we get the estimate $|E_k\cap E^{\mathsf c}|\leq \delta^{-1}(1-\alpha)|E_k|$ for $\delta\in(1-\alpha,1)$. This allows us to estimate
\[
|A_R|=\Abs{\bigcup_k E_k \cap E^{\mathsf c}}\leq \frac{1-\alpha}{\delta}\Abs{ \bigcup_k E_k}=\frac{1-\alpha}{\delta} \Abs{\bigcup_{Q\in\mathcal Q_R} Q}\;.
\]
Using Lemma~\ref{l.growth} we can conclude that for each $R$ we have
\[
\begin{split}
 \frac{w(A_R)}{w(3R)}& \leq 2 \big(\frac{|A_R|}{|R|}\big)^{(c_n[w]_{A_\infty})^{-1}} \leq 2 \Big(\frac{1-\alpha}{\delta}\frac{\abs{\cup_{Q\in\mathcal Q_R} Q} }{|R|} \Big)^ { (c_n [w]_{A_\infty})^{-1}}
\\
&\lesssim_n \big(\frac{1-\alpha}{\delta} \big)^ { (c_n [w]_{A_\infty})^{-1}}.
\end{split}
\]
Now we can sum for $R\in\mathcal R$ to get
\[
\begin{split}
II & \lesssim_n \sum_{R\in\mathcal R} \big(\frac{1-\alpha}{\delta} \big)^ { (c_n [w]_{A_\infty})^{-1}}w(3R)\lesssim_n \Delta_w ^2 \big(\frac{1-\alpha}{\delta} \big)^ { (c_n [w]_{A_\infty})^{-1}} \sum_{R\in\mathcal R} w(R)
\\
&\leq  \Delta_w ^2 \big(\frac{1-\alpha}{\delta} \big)^ { (c_n [w]_{A_\infty})^{-1}} \big(1-2(1-\alpha)^{(c_n[w]_{A_\infty})^{-1}}\big)^{-1}  \sum_{R\in\mathcal R} w(R\cap E)
\end{split}
\]
whenever $\alpha>1-e^{-C_n[w]_{A_\infty}}$, where in the last estimate we have used \eqref{e.RcapE}. Using the fact that the Vitali cubes $R\in\mathcal R$ are pairwise disjoint we finally get that for $\alpha>1-e^{-C_n[w]_{A_\infty}}$ we have
\[
 II \lesssim_n  \Delta_w ^2 \big(\frac{1-\alpha}{\delta} \big)^ { (c_n [w]_{A_\infty})^{-1}} \big(1-2(1-\alpha)^{(c_n[w]_{A_\infty})^{-1}}\big)^{-1}w(E).
\]
Choosing $\delta\coloneqq (1-\alpha)^\frac{1}{2}>1-\alpha$ and summing up the estimates for $I$ and $II$ we conclude
\[
 w(K\setminus E)\lesssim_n \Delta_w ^2 (1-\alpha)^{(c_n[w]_{A_\infty})^{-1}}\quad\text{whenever}\quad \alpha>1-e^{-C_n[w]_{A_\infty}}.
\]
which easily implies the desired estimate. Note that by taking the maximum of $c_n, C_n$ we can assume that the same dimensional constant appears in both places of the estimate above.
\end{proof}
\begin{remark} If one is not interested in the constants in the statement of Proposition~\ref{p.weighted} then a very easy proof is available. Indeed, one just needs to note that
\[
\{ \M\chi_E>\alpha\}\subseteq \big\{ \M_w(\chi_E)>1-2 (1-\alpha)^{(c_n[w]_{A_\infty} )^{-1}}\big\}
\]
when $\alpha$ is sufficiently close to $1$ and the result follows by Theorem~\ref{t.Soldoubling} since $w\in A_\infty$. Note however that the precise exponent in the statement of Theorem~\ref{t.main2} plays a crucial role in the embedding of $A_\infty$ into $A_p$ below.\qed
\end{remark}
\begin{remark} In one dimension it is possible to get the estimate of Proposition~\ref{p.weighted} above without the term $\Delta_w ^2$, and the proof is elementary. Indeed, Let $K\subseteq E_\alpha\coloneqq\{\M\chi_E>\alpha\}$ where $K$ is compact. Arguing as in~\S\ref{s.easy} we can find finitely many intervals $I_j$ that cover $K$ and such that $\sum_j \chi_{I_j}\leq 2$. Then for $w\in A_\infty$ we have
	\[
	w(K)\leq w(E)+\sum_{j}\frac{w(I_j\cap E^{\mathsf c})}{w(I_j)}w(I_j).
	\]
Using Lemma~\ref{l.growth} and the fact that the $I_j$'s have overlap at most $2$ we get
\[
w(E_\alpha)\leq w(E)+2 (1-\alpha)^{(c[w]_{A_\infty})^{-1}}w(E_\alpha).
\]
Thus in dimension $n=1$ we thus have the improved estimate
\[
\Cw(\alpha)-1 \leq 2(1-\alpha)^{(c[w]_{A_\infty})^{-1}}\quad\text{whenever}\quad \alpha>1-e^{-c[w]_{A_\infty}} .
\]
Note that the $[w]_{\infty}$ constant is the only information needed to write this estimate, in contrast with the higher dimensional case where the doubling constant of $w$ is also needed. We do not know however if the appearance of the doubling constant is just an artifact of the proof. \qed
\end{remark}
The Solyanik estimates given above characterize $A_\infty$ in terms of the constants involved. One could argue that this is a very complicated way to describe $A_\infty$. However, there is an advantage, namely that these estimates imply a quantitative embedding of $A_\infty$ into $A_p$. Furthermore, when one writes down this embedding then the roles of the constants in a Solyanik estimate become more transparent.
\begin{proof}[Proof of Theorem~\ref{t.second}] We assume that $w\in A_\infty$ and thus the estimate
\[
 \Cw(\alpha)-1 \lesssim_n \Delta_w ^2 (1-\alpha)^{(c_n[w]_{A_\infty})^{-1}}\quad\text{whenever}\quad \alpha\geq 1-e^{-c_n[w]_{A_\infty}}
\]
is available, where $c_n$ is some constant depending only on the dimension. Let as set $\alpha_o\coloneqq 1-e^{-c_n[w]_{A_\infty}}$ so in particular we have that $\Cw(\alpha_o)\leq1+C_n\Delta_w ^2\lesssim_n \Delta_w ^2$. Here we remember that, by (i) of Corollary~\ref{c.growth} we have $\Delta_w\lesssim \exp(\exp(c_n[w]_{A_\infty}))$ for some dimensional constant $c_n>0$. Perusing the proof of \cite{hlp}*{Theorem 6.1} one then sees that:	
\[
w(\{\M\chi_E>\lambda\})\leq \exp{\bigg[\log \Cw(\alpha_o) \bigg( \bigg\lceil \frac{-\log\frac{\alpha_o}{\lambda}}{\log \alpha_o}  \bigg \rceil \bigg\lceil 2+ \frac{\log^+ (2^n\alpha_o)}{\log 1/\alpha_o} \bigg\rceil +1 \bigg)\bigg]}w(E),
\]
where $\lceil x\rceil$ denotes the smallest positive integer which is no less than $x$. Assuming as we may that $\alpha_o 2^n>1$ and using the estimate $\lceil x\rceil \leq x+1$ we get
\[
w(\{x\in\R^n:\M\chi_E>\lambda)\})\leq \Cw(\alpha) \frac{w(E)}{\lambda^{p_o}}
\]
for $p_o=e^{C_n[w]_{A_\infty}} \log \Cw(\alpha_o)$, where $C_n>1$ is a dimensional constant. That is, $\M$ is of restricted weak type $(p_o,p_o)$ with respect to $w$. Marcinkiewicz interpolation now gives that $\M$ is bounded on $L^p(w)$ for $p>e^{C_n[w]_{A_\infty} }$ and we have the norm estimate
\[
\|\M\|_{L^p(w)\to L^p(w)} \leq 2 \frac{p^\frac{1}{p}\Cw(\alpha_o) ^{\frac{1}{p} e^{C_n[w]_{A_\infty}}}} {(p-e^{C_n[w]_{A_\infty} })^\frac{1}{p}}.
\]
Observe that, setting $q = 2 e^{C_n[w]_{A_\infty}}$, we have
  \[
\|\M\|_{L^q(w)\to L^q(w)} \leq 4 (\Cw(\alpha_o) )^\frac{1}{2}.
\]
By the Riesz-Thorin theorem, applied to any linearization of $\M$, we then have that
\[
\|\M\|_{L^p(w)\to L^p(w)} \leq 4 ^\frac{q}{p}( \Cw(\alpha_o))^{q/2p}
\]
for $p > q$. Since we always have $\|\M\|_{L^p(w)\to L^p(w)} \geq [w]_{A_p} ^\frac{1}{p}$ for all $p>1$ we get that for some dimensional constant $c_n>1$ we have
\[[w]_{A_p} \leq  (16\Cw(\alpha_o)) ^{e^{c_n[w]_{A_\infty}}} \leq  e^{e^{c_{n}[w]_{{A}_\infty}}}\]
for $p > e^{c_n[w]_{A_\infty}}$ as we wanted.
\end{proof}
Recall here that the Hru{\v{s}}{\v{c}}ev $A_\infty$ constant is defined by
\[
[w]_{A_\infty} ' \coloneqq \sup_Q \bigg(\frac{1}{|Q|}\int_Q w\bigg)  \exp\bigg(\frac{1}{|Q|}\int_Q \log w^{-1} \bigg)
\]
and we have $w\in A_\infty \Leftrightarrow [w]_{A_\infty}' <+\infty$. Furthermore it is known that $[w]_{A_\infty}\lesssim_n [w]_{A_\infty}'$; see for example \cite{HytP}. Remember here that the $[w]_{A_\infty} '$ is the natural endpoint constant which can be defined by taking the formal limit, as $p\to \infty$, of the constants $[w]_{A_p}$. Theorem~\ref{t.second} now immediately implies:

\begin{corollary}\label{c.constants} Let $[w]_{A_\infty} '$ be the Hru{\v{s}}{\v{c}}ev constant, defined above, and $[w]_{A_\infty}$ denote the Fujii-Wilson constant as usual. We have
	\[
	[w]_{A_\infty} \lesssim_n [w]_{A_\infty}' \leq e^{e^{c_n[w]_{A_\infty}}}
	\]
for some dimensional constant $c_n>1$.
\end{corollary}
We note that it was shown in \cite{HytP} that there are examples of weights such that $[w]_{A_\infty} '$ is exponentially larger than $[w]_{A_\infty}$.  Moreover, as was shown in \cite{BR}, the doubly exponential estimate of $[w]_{A_\infty '}$ in terms of $[w]_{A_\infty}$ is sharp, up to dimensional constants. See in particular \cite{BR}*{Theorem 1.3} and note that the the \emph{reverse H\"older constant} $[w]_{\mathsf{RH}_1}$ in that paper satisfies $[w]_{\mathsf{RH}_1}\eqsim_n [w]_{A_\infty}$.

 We can also observe that the doubling constant of $w$ is known to satisfy
\[
\Delta_w \lesssim_n ([w]_{A_\infty}  ' )^{2n}.
\]
This estimate is proved in \cite{Kor}. On the other hand the only estimate for the doubling constant of $w$ with respect to $[w]_{A_\infty}$ that we are aware of is doubly exponential in $[w]_{A_\infty}$:
\[
\Delta_w \leq e ^{e^{c_n[w]_{A_\infty}}}.
\]
These estimates are consistent with the estimate of Corollary~\ref{c.constants}.


\section*{Acknowledgments} Part of this work was carried out while the authors were visiting the Department of Mathematical Analysis at the University of Seville. We are indebted to Teresa Luque and Carlos P\'erez for their warm hospitality. We wish to  thank Teresa Luque and Alex Stokolos for pointing out some important references related to the subject of this paper.   Finally, we would like to thank the referee for an expert reading that resulted in many improvements throughout the paper.


\end{document}